\documentclass{amsart}

\usepackage{a4,amsmath,amssymb,amscd,verbatim}

\newtheorem{proposition}[equation]{Proposition}
\newtheorem{theorem}[equation]{Theorem}
\newtheorem{corollary}[equation]{Corollary}
\newtheorem{lemma}[equation]{Lemma}

\theoremstyle{definition}

\newtheorem{definition}[equation]{Definition}

\newtheorem{remark}[equation]{Remark}

\numberwithin{equation}{section}

\begin{document}
\bibliographystyle{plain} \title[Amenability and critical exponents]
{Amenability, critical exponents of subgroups and growth of closed geodesics} 

\author{Rhiannon Dougall}
\address{Mathematics Institute, University of Warwick,
Coventry CV4 7AL, U.K.}
\email{R.Dougall@warwick.ac.uk}
\author{Richard Sharp} 
\address{Mathematics Institute, University of Warwick,
Coventry CV4 7AL, U.K.}
\email{R.J.Sharp@warwick.ac.uk}


\keywords{}


\begin{abstract}
Let $\Gamma$ be a (non-elementary)
convex co-compact group of isometries of a pinched Hadamard manifold $X$.
We show that a normal subgroup $\Gamma_0$ has critical exponent equal to the critical 
exponent of 
$\Gamma$ if and only if $\Gamma/\Gamma_0$ is amenable. We prove a similar result 
for the exponential growth rate of closed geodesics on $X/\Gamma$. 
These statements are analogues of classical results of Kesten for random walks on groups 
and Brooks for the spectrum of the Laplacian on covers of Riemannian manifolds.
\end{abstract}

\maketitle

%
%
%
%
%
%
%
%
%

\section{Introduction}\label{in}

Let $X$ be a connected simply connected and complete
Riemannian manifold with sectional curvatures bounded between two negative constants. 
We call such an $X$ a {\it pinched Hadamard manifold}. Let $\Gamma$ be a {\it non-elementary}
and
{\it convex co-compact}
group of isometries of $X$ (see section 2 for precise definitions) and let 
$\Gamma_0 \triangleleft \Gamma$ be a normal subgroup. We write $G= \Gamma/\Gamma_0$.
We define the {\it critical exponent}
of $\Gamma$ to be the abscissa of convergence of the Dirichlet series
\[
\sum_{g \in \Gamma} e^{-sd_X(o,go)},
\]
for any choice of base point $o \in X$, and denote it by $\delta(\Gamma)$. The critical exponent
$\delta(\Gamma_0)$
of $\Gamma_0$ is defined in the same way. The fact that $\Gamma$ is non-elementary 
and convex co-compact means that $\delta(\Gamma)>0$ and it is clear that $\delta(\Gamma_0)
\leq \delta(\Gamma)$. It is natural to ask when we have equality
and our main result will give a precise answer to this question, which will depend only on $G$ as an abstract group. (We will discuss the history of this and related problems in the next section.)

Before stating our result, we introduce an alternative formulation. Consider
the quotient manifolds $M= X/\Gamma$ and $\widetilde M = X/\Gamma_0$;
$\widetilde M$ is a regular cover of $M$ with covering group $G$.
Then $M$ has a countably infinite set 
$\mathcal C(M)$ of closed geodesics (which are not assumed to be prime). 
For $\gamma \in \mathcal C(M)$, 
we write $l(\gamma)$ for its length. For each $T>0$, the set 
$\{\gamma \in \mathcal C(M) \hbox{ : } l(\gamma) \leq T\}$ is finite and we can  define a
 number $h=h(M)>0$ by
\[
h(M) = \lim_{T \to \infty} \frac{1}{T} \log 
\#\{\gamma \in \mathcal C(M) \hbox{ : } l(\gamma) \leq T\}.
\]
(The limit exists and, in fact, 
$\lim_{T \to \infty} T e^{-h(M)T} \#\{\gamma \in \mathcal C(M) \hbox{ : } l(\gamma) \leq T\} =1/h(M)$
\cite{PP, Perry}.)
Similarly, we write $\mathcal C(\widetilde M)$ for the set of closed geodesics on $\widetilde M$.
If $G$ is infinite then, for a given $T$, the set 
$\{\gamma \in \mathcal C(\widetilde M) \hbox{ : } l(\gamma) \leq T\}$ is infinite
(since a single closed geodesic has infinitely many
images under the action of $G$). However, we can obtain a finite quantity by 
choosing any relatively compact open subset $W$ of the unit-tangent
bundle $S\widetilde M$ 
that intersects 
the non-wandering set for the geodesic flow
and considering
the set
\[
\mathcal C(\widetilde M,W) = \{\gamma \in \mathcal C(\widetilde M) \hbox{ : } 
\hat \gamma \cap W \neq \varnothing\},
\]
where $\hat \gamma$ is the periodic orbit for the geodesic flow lying over the closed geodesic
$\gamma$.
Following \cite{PPS}, we then define
 $h(\widetilde M) \leq h(M)$
by
\[
h(\widetilde M) = \limsup_{T \to \infty} \frac{1}{T} \log\#\{\gamma \in \mathcal C(\widetilde M,W) \hbox{ : }
l(\gamma) \leq T\}
\]
and this is independent of the choice of $W$.
Again we may ask when we have equality.

It is well known that $h(M) = \delta(\Gamma)$, and both are equal to the topological 
entropy of the geodesic flow over $X/\Gamma$. Furthermore, the work 
of Paulin, Pollicott and Schapira in 
\cite{PPS} shows that
$h(\widetilde M) = \delta(\Gamma_0)$ (and that the limsups in the definitions are, in fact, limits). 
We will discuss this further in section 2.

Our main result is the following.

\begin{theorem}\label{main} 
Let $\Gamma$ be a convex co-compact group of isometries of a pinched Hadamard manifold
$X$ and let $\Gamma_0$ be a normal subgroup of $\Gamma$. Then the following are equivalent:
\begin{enumerate}
\item[(i)]
$\delta(\Gamma_0) = \delta(\Gamma)$,
\item[(ii)]
$h(X/\Gamma_0) = h(X/\Gamma)$,
\item[(iii)]
$G = \Gamma/\Gamma_0$ is amenable.
\end{enumerate}
\end{theorem} 

In view of the equivalence of (i) and (ii) discussed above, this will follow from Theorem \ref{mainresultforclosedgeodesics} 
below, in which we prove the equivalence of (ii) and (iii).

\begin{remark}
In fact, that (iii) implies (i) is 
a theorem of Roblin \cite{Roblin} (see also the expository
account in
\cite{RoblinTapie}), which actually applies in the more general situation where $X$ is a CAT($-1$) space.
\end{remark}

\begin{remark}  In the special case where $X = \mathbb H^{d+1}$ and 
$\delta(\Gamma)>d/2$, the
 statement
$\delta(\Gamma_0) = \delta(\Gamma)$ if and only if $G$ is amenable is a result of Brooks
\cite{Brooks85} (and holds when $\Gamma$ is geometrically finite, which is a more general
condition than convex co-compactness). We discuss this in more detail in the next section. If $X = 
\mathbb H^{d+1}$
and $\Gamma$ is essentially free then Stadlbauer showed the same result holds without the 
assumption
$\delta(\Gamma)>d/2$ \cite{Stad}. The class of essentially free groups includes
all non-co-compact geometrically finite Fuchsian groups (i.e. discrete groups of isometries
of $\mathbb H^2$) and all Schottky groups.
\end{remark}

We conclude the introduction by outlining the contents of the paper. In the next section, we define
the key concepts associated to groups that are mentioned above
and discuss some history of this and analogous problems. Our approach to Theorem 
\ref{main} is via dynamics. More precisely, we consider the geodesic flow over $M$ and $\widetilde 
M$ and a class of symbolic dynamical systems that model them. These symbolic systems
belong to a class called countable state Markov shifts: we introduce these in section 3 and define
a key quantity called the Gurevi\v{c} pressure. We also mention recent results of Stadlbauer
that will be key to our analysis.
In sections 4, 5 and 6, we consider the geodesic flows over $M$ and $\widetilde M$ and discuss
how they may be modeled by symbolic systems, particular using a skew product construction to 
record information about lifts to the cover. Finally, in section 7, we link various zeta functions 
defined by the 
closed geodesics to the Gurevi\v{c} pressure and hence, using Stadlbauer's result, prove that 
the equality of $h(M)$ and $h(\widetilde M)$ is equivalent to amenability of the covering group.

\section{Background and History}

Let $X$ be a pinched Hadamard manifold, i.e.~a connected simply connected complete 
Riemannian manifold such that its sectional curvatures lie in a interval $[-\kappa_1,-\kappa_2]$,
for some $\kappa_1>\kappa_2>0$. 
Associated to $X$ is a well defined topological space $\partial X$ called the Gromov boundary.
This can be defined to be the set of equivalence classes of geodesic rays emanating
from a fixed base point, where two rays are equivalent it their distance apart is bounded above.
Let $\Gamma$ be a group of isometries acting freely and 
properly discontinuously on $X$. We say that $\Gamma$ is {\it non-elementary} if 
it is not a finite extension of a cyclic group. Fix $o \in X$. Then the orbit
$\Gamma o = \{go \hbox{ : } g \in \Gamma\}$ accumulates only on $\partial X$ and we call 
the set of accumulation points $L_\Gamma$ the {\it limit set} of $\Gamma$; this is independent
 of the 
choice of
$o$. Let $C(\Gamma)$ denote the intersection of $X$ with the convex hull 
(with respect to the metric on $X$) of $L_\Gamma$. We say that $\Gamma$ is 
{\it convex co-compact} if $C(\Gamma)/\Gamma$
is compact.
If $\Gamma$ is convex co-compact then $M = X/\Gamma$ has a compact core:
a manifold with boundary $M_0$ which
contains $\mathcal C(M)$.

Now suppose that $\Gamma_0 \triangleleft \Gamma$ and that $\widetilde M = X/\Gamma_0$.
Then $\widetilde M$ is a regular $G = \Gamma/\Gamma_0$-cover of $M$.
Let $\pi : \widetilde M \to M$ denote the projection.
It is shown in \cite{PPS} that if $W$ is a relatively compact open subset of 
the unit-tangent bundle $S\widetilde M$ intersecting the non-wandering set for the geodesic flow
$\widetilde\phi_t : S\widetilde M \to S\widetilde M$
for any $c>0$,
\[
\delta(\Gamma_0) = \lim_{T\to \infty} \frac{1}{T} \log
\#\{\gamma \in \mathcal C(\widetilde M,W) \hbox{ : }
T-c < l(\gamma) \leq T\}.
\]
(To see this, take $F=0$ in Theorem 1.1 of \cite{PPS}.)

\begin{lemma} \label{equalityofhanddelta}
We have $h(M) = \delta(\Gamma_0)$.
\end{lemma}

\begin{proof}
If $\delta(\Gamma_0)>0$ when we may replace the condition $l(\gamma) \leq T$ 
with $T-c < l(\gamma) \leq T$ without affecting the exponentional growth rate and conclude the result.
On the other hand, if $\delta(\Gamma_0)=0$
then the simple inequality $h(\widetilde M) \leq \delta(\Gamma_0)$ gives $h(\widetilde M)=0$.
\end{proof}

A (countable) group $G$ is {\it amenable} if $\ell^\infty(G,\mathbb R)$ admits an invariant mean, 
i.e. that there exists
a bounded linear functional $\nu : \ell^\infty(G,\mathbb R) \to \mathbb R$ such that, for all
$f \in \ell^\infty(G,\mathbb R)$, 
\begin{enumerate}
\item
$\inf_{g \in G} f(g) \leq \nu(f) \leq \sup_{g \in G} f(g)$; and
\item
for all $g \in G$, $\nu(f_g)=\nu(f)$, where $f_g(h) = f(g^{-1}h)$.
\end{enumerate}
The concept was introduced by von Neumann in 1929. One sees immediately from this 
definition that finite groups are amenable by taking
\[
\nu(f) = \frac{1}{|G|} \sum_{g \in G} f(g).
\]

An alternative criterion for amenability was given by F\o{}lner \cite{Fol}:
$G$ is amenable if and only if, for every $\epsilon>0$ and every finite set $\{g_1\ldots,g_n\}
\subset G$, there exists a finite set $F \subset G$ such that
$\#(F \cap g_iF) \geq (1-\epsilon) \#F$, $i=1,\ldots,n$.
Using this criterion, it is easy to see that abelian groups are amenable and that, more 
generally, groups
with subexponential growth are amenable \cite{HT}. 
Furthermore, since amenability is closed under extensions, solvable groups are amenable. In 
particular, there are amenable groups with exponential growth (e.g. lamplighter groups). On the 
other hand, a group containing the free group on two generators is not amenable
and non-elementary Gromov hyperbolic groups (a class which includes the convex co-compact
groups above) are not amenable.

There are numerous results that connect growth and spectral properties of groups and manifolds to 
amenability. The prototype is the following theorem of Kesten from 1959. 
Let $G$ be a countable group 
and let $p : G \to \mathbb R^+$ be a symmetric probability distribution 
(i.e. $\sum_{g \in G} p(g) = 1$ and $p(g^{-1}) =p(g)$ for all $g \in G$) such that 
its support,
$\mathrm{supp}(p)$, generates $G$. This defines a symmetric random walk on $G$ with 
transition probabilities $P(g,g') = p(g^{-1}g')$. If we define $\lambda(G,P)$ to be the 
$\ell^2(G)$-spectral radius of $P$ then we have
\[
\lambda(G,P) = \limsup_{n \to \infty} P^n(g,g)^{1/n} 
= \lim_{n \to \infty} P^{2n}(g,g)^{1/2n},
\]
for any $g \in G$.
It is clear that $\lambda(G,P) \leq 1$.

\begin{theorem}[Kesten \cite{Kes}]\label{kesten}
We have $\lambda(G,P)=1$ if and only if $G$ is amenable.
\end{theorem}

Note that, while $\lambda(G,P)$ depends on the probability $p$, 
whether or not it takes the value $1$ 
depends only on the group $G$. 

Subsequently, results inspired by Kesten's Theorem were obtained in 
a variety of other situations. In the setting of group theory, 
the most notable result is Grigorchuk's co-growth criterion \cite{Gri} (see also
 Cohen \cite{Cohen}) for finitely generated groups.
 Recall that a finitely generated group $G$ may be written as $F/N$, where $F$ is a free 
 group of rank $k$ and $N$ is a normal subgroup. If $|\cdot|$ denotes the word length on $F$
 with respect to a free generating set
 then $\lim_{n \to \infty} (\#\{x \in F \hbox{ : } |x| =n\})^{1/n} = 2k-1$.
 
 \begin{theorem} [Grigorchuk \cite{Gri}]
 We have
 \[
 \limsup_{n \to \infty} (\#\{x \in N \hbox{ : } |x| =n\})^{1/n} = 2k-1
 \]
 if and only if $G$ is amenable.
 \end{theorem}
 
Subsequently, various extensions of this to graphs and (non-backtracking) random 
walks were obtained by Woess \cite{woess},
 Northshield \cite{north1, north2}
and Ortner and Woess \cite{ortnerwoess}. 

In the setting of Riemannian manifolds, an analogue is provided by 
the following spectral result of Brooks.
Let $M$ be a 
complete Riemannian manifold 
which is of ``finite topological type'', i.e. that it is topologically the union of finitely many simplices,
and let $\widetilde M$ be a regular covering of 
$M$ with covering group $G$. Let $\lambda_0(M)$ and
$\lambda_0(\widetilde M)$ denote the infimum of
the spectrum of the Laplace-Beltrami operator on $M$ and $\widetilde M$, respectively; 
then $\lambda_0(\widetilde M) \geq \lambda_0(M)$. Brooks showed that amenability of $G$
implied equality and that, together with an additional 
condition, the converse holds. More precisely, he proved the following.

\begin{theorem}[Brooks \cite{Brooks85}]\label{brooks}

\noindent
(i) If $G$ is amenable then
$\lambda_0(\widetilde M)=\lambda_0(M)$.

\noindent
(ii) Let $\phi$ be the lift of a $\lambda_0(M)$-harmonic function to $\widetilde M$
and let $F$ be a fundamental region for $M$ on $\widetilde M$. Suppose that there is a compact 
$K \subset F$ such that
\[
\inf_S \frac{\int_S \phi^2 \, d\mathrm{Area}}{\int_{\mathrm{int}(S)} \phi^2 \, d\mathrm{Vol}} >0,
\]
where the infimum is taken over co-dimension $1$ submanifolds 
$S$
that divide $F\setminus K$ into an 
interior and an exterior. If $\lambda_0(\widetilde M)=\lambda_0(M)$
then $G$ is amenable.
\end{theorem}

(See also Brooks \cite{Brooks81} for the case when $M$ is compact and $G$ is its
fundamental group and Burger
\cite{Burger} for a shorter proof.)
The Cheeger-type condition in part (ii) holds if, for example, $M$ is a convex 
co-compact quotient of the $(d+1)$-dimensional
real hyperbolic space $\mathbb H^{d+1}$ and $\lambda_0(M) < d^2/4$.

The problem of critical exponents was also first considered in the case $X = \mathbb H^{d+1}$.
In the early 1980s, Rees \cite{Rees}
showed that if $\Gamma/\Gamma_0$ is abelian then we have equality 
(and her dynamical methods generalize to variable curvature). 
Soon afterwards, Brooks obtained a more general result as a corollary of Theorem \ref{brooks}.
This is due to the fact that, for $X= \mathbb H^{d+1}$,
$\delta(\Gamma)$ and $\lambda_0(M)$ are related by the formula 
$\lambda_0(M) = \delta(\Gamma)(d-\delta(\Gamma))$, provided $\delta(\Gamma) > d/2$, with the same holding for $\delta(\Gamma_0)$ and $\lambda_0(\widetilde M)$.
In particular, if $\delta(\Gamma)>d/2$ then $\lambda_0(M) < d^2/4$
and so Theorem \ref{brooks}
implies the statement that $\delta(\Gamma_0) = \delta(\Gamma)$ if and only if $G$ is amenable.

More recently, Stadlbauer \cite{Stad} and Jaerisch \cite{Jaerisch} have considered
the relation between amenability and certain growth rates that occur in the study of skew product 
extensions of dynamical systems. It will be clear below that we are greatly indebted to this 
work in our analysis.

\section{Countable State Markov Shifts and Gurevi\v{c} Pressure}

In this section we will define countable state Markov shifts and discuss some of their properties.
Basic definitions and results are taken from chapter 7 of \cite{kitchens}.
In the rest of the paper, we shall be concerned with finite state shifts and
skew product extensions of these by
a countable group, so we shall often specialise to these two cases.

Let $S$ be a countable set, called the {\it alphabet}, and let $A$ be a matrix, 
called the transition matrix, indexed by 
$S \times S$ with entries zero
or one. We then define the space
\[
\Sigma_A^+ = \left\{x = (x_n)_{n=0}^\infty \in S^{\mathbb Z^+} \hbox{ : }
A(x_n,x_{n+1})=1 \ \forall n \in \mathbb Z^+\right\},
\]
with the product topology induced by the discrete topology on $S$. This topology is compatible
with the metric $d(x,y)= 2^{-n(x,y)}$, where
\[
n(x,y) = \inf\{n \hbox{ : } x_n \neq y_n\},
\]
with $n(x,y)=\infty$ if $x=y$. 
If $S$ is finite then $\Sigma_A^+$ is compact.
We say that $A$ is {\it locally finite} if all its row and column sums are finite.
Then $\Sigma_A^+$ is locally compact if and only if $A$ is locally finite.
(The skew product extensions we consider have this latter property.)

We define the (one-sided) countable state topological Markov shift
$\sigma : \Sigma_A^+ \to \Sigma_A^+$ by
$(\sigma x)_n = x_{n+1}$.  This is a continuous map.
We will say that $\sigma$ is {\it topologically
transitive} if it has a dense orbit and {\it topologically mixing} if, given non-empty 
open sets $U,V \subset \Sigma_A^+$, there exists
$N \geq 0$ such that $\sigma^{-n}(U) \cap V \neq \varnothing$ for all $n \geq N$. 
We say that the matrix $A$ is {\it irreducible} if, for each $(i,j) \in S \times S$, there exists
$n =n(i,j)\geq 1$ such that $A^n(i,j)>0$.
For $A$ irreducible, set $p\geq 1$ to be the greatest common divisor of periods of periodic orbits
$\sigma : \Sigma_A^+ \to \Sigma_A^+$; this $p$ is called the period of $A$.
We say that $A$ is {\it aperiodic} if $p=1$ or, equivalently,
if there exists $n \geq 1$ such that $A^n$ has all entries positive.
Suppose that $A$ is locally finite. Then $\sigma : \Sigma_A^+ \to \Sigma_A^+$ is topologically
transitive if and only if $A$ is irreducible and $\sigma : \Sigma_A^+ \to \Sigma_A^+$ is topologically
mixing if and only if $A$ is aperiodic.

Suppose that $A$ is irreducible but not aperiodic and fix $i \in S$. Then we may partition
$S$ into sets $S_l$, $l=0,\ldots,p-1$, defined by
\[
S_l = \{j \hbox{ : } A^{np+l}(i,j)>0 \mbox{ for some } n \geq 1\}.
\]
(This partition is independent of the choice of $i$.)
For each $l$, let $A_l$ denote the restriction of $A$ to $S_l \times S_l$; then 
$\sigma : \Sigma_{A_l}^+ \to \Sigma_{A_{l+1}}^+$ (mod $p$) and $A_l^p$
is aperiodic.

We say that an $n$-tuple
$w = (w_0,\ldots,w_{n-1}) \in S^n$ is an {\it allowed word} of length $n$ if
$A(w_j,w_{j+1})=1$ for $j =0,\ldots,n-2$. 
We will write $\mathcal W^n$ for the set of allowed words of length $n$.
If $w \in \mathcal W^n$ then we define 
the associated cylinder set $[w]$ by 
\[
[w] = \{x \in \Sigma_A^+ \hbox{ : } x_j =w_j, \ j=0,\ldots,n-1\}.
\]

For a function $f : \Sigma_A^+ \to \mathbb R$, set
\[
V_n(f) = \sup\{|f(x)-f(y)| \hbox{ : } x_j=y_j, \ j=0,\ldots,n-1\}.
\]
We say that $f$ is {\it locally H\"older continuous} if there exist $0<\theta<1$ and $C \geq 0$ such that, for all $n \geq 1$,
$V_n(f) \leq C\theta^n$.
(There is no requirement of $V_0(f)$ and a locally H\"older $f$ may be unbounded.)
For $n \geq 1$, we write
\[
f^n := f + f \circ \sigma + \cdots + f \circ \sigma^{n-1}.
\]

\begin{definition} Suppose that $\sigma : \Sigma_A^+ \to \Sigma_A^+$ is topologically
transitive and let 
$f : \Sigma_A^+ \to \mathbb R$ be a locally H\"older continuous function. Following 
Sarig \cite{sarig-etds},
we define the
{\it Gurevi\v{c} pressure}, $P_G(\sigma,f)$, of $f$ to be
\[
P_G(\sigma,f) = \limsup_{n \to \infty} \frac{1}{n} \log 
\sum_{\substack{\sigma^n x=x \\ x_0=a}}
e^{f^n(x)},
\]
where $a \in S$. (The definition is independent of the choice of $a$.)
\end{definition}

\begin{remark}
In \cite{sarig-etds}, Sarig gives this definition in the case where
 $\sigma : \Sigma_A^+ \to \Sigma_A^+$ is topologically mixing.
 However, the above decomposition of $\Sigma_A^+ = \Sigma_{A_0}^+ \cup \cdots
 \cup \Sigma_{A_{p-1}}^+$, with $\sigma^p$ topologically mixing on each component,
 together with the regularity of the function $f$,
 shows that the same definition may be made in the topologically transitive case.
 \end{remark}
 
We now specialise to the case where $S$ is finite. In this situation, we call 
$\sigma : \Sigma_A^+ \to \Sigma_A^+$
a (one-sided) subshift of finite type. The above definitions and results hold. If $f : \Sigma_A^+
\to \mathbb R$ is H\"older continuous then $f$ is locally H\"older. Provided
$\sigma : \Sigma_A^+ \to \Sigma_A^+$ is topologically transitive, the Gurevi\v{c} pressure
$P_G(\sigma,f)$ agrees with the standard pressure $P(\sigma,f)$, defined by
\[
P(\sigma,f) = \limsup_{n \to \infty} \frac{1}{n} \log 
\sum_{\sigma^n x=x}
e^{f^n(x)}
\]
and if $\sigma$ is topologically mixing then the $\limsup$ may be replaced 
with a limit.

We now consider skew product extensions of a shift of finite type 
$\sigma : \Sigma_A^+ \to 
\Sigma_A^+$, which we will assume to be
topologically mixing. Let $G$ be a countable group and let 
$\psi : \Sigma_A^+ \to G$ be a function
depending only on two co-ordinates, $\psi(x)=\psi(x_0,x_1)$. 
(One may consider more 
general $\psi$ but this set-up suffices for our needs.) This data defines a 
{\it skew product
extension} $\widetilde \sigma : \Sigma_A^+ \times G \to \Sigma_A^+ \times G$ by
$\widetilde \sigma(x,g) = (\sigma x,g\psi(x))$. For $n \geq 1$ define
$\psi_n$ by
\[
\psi_n(x) =   \psi(x) \psi(\sigma x)\cdots \psi(\sigma^{n-1}x);
\]
then $\widetilde \sigma^n(x,g) = (x,g)$ if and only if
$\sigma^n x=x$ and $\psi_n(x)=1_G$, where $1_G$ is the identity element
in $G$.

The map $\widetilde \sigma : \Sigma_A^+ \times G \to \Sigma_A^+ \times G$ 
is itself a countable state Markov shift
with alphabet $S \times G$ and transition matrix $\widetilde A$ defined by
$\widetilde A((i,g),(j,h)) = 1$ if $A(i,j)=1$ and $\psi(i,j) = g^{-1}h$ and 
$\widetilde A((i,g),(j,h)) = 0$
otherwise.
Clearly, $\widetilde A$ is locally finite and so the topological transitivity and 
topological 
mixing of $\widetilde \sigma$ are equivalent to $\widetilde A$ being irreducible and 
aperiodic,
respectively.

Let $f : \Sigma_A^+ \to \mathbb R$ be H\"older continuous and define 
$\widetilde f : \Sigma_A^+ \times G \to \mathbb R$ by $\widetilde f(x,g) = f(x)$;
then $\widetilde f$ is locally H\"older continuous and its Gurevi\v{c} pressure 
$P_G(\widetilde \sigma,\widetilde f)$ is defined. In fact,
it is easy to see that, due to the mixing of $\sigma$,
\[
P_G(\widetilde \sigma,\widetilde f) = \limsup_{n \to \infty} \frac{1}{n} \log 
\sum_{\substack{\sigma^n x=x \\ \psi_n(x)=1_G}}
e^{f^n(x)}.
\]
It is clear that $P_G(\widetilde \sigma,\widetilde f) \leq P(\sigma,f)$ and it is interesting to ask
when equality holds. Stadlbauer has shown this depends only on the amenability of
the group $G$,
provided the skew product and
the function $f$ satisfy appropriate symmetry conditions, which we now describe.

Suppose there is a fixed point free involution $\kappa : S \to S$ such that
$A(\kappa j,\kappa i) = A(i,j)$, for all $i,j \in S$. We say that 
the skew product $\widetilde \sigma : \Sigma_A^+ \times G \to \Sigma_A^+ \times G$ is
{\it symmetric} (with respect to $\kappa$) if $\psi(\kappa j,\kappa i) = \psi(i,j)^{-1}$.
A function $f : \Sigma_A^+ \to \mathbb R$ is called {\it weakly symmetric} if, for all
$n \geq 1$ and and all length $n$ cylinders $[z_0,z_1,\ldots,z_{n-1}]$, there exists
$D_n >0$ such that $\lim_{n \to \infty} D_n^{1/n}=1$ and
\[
\sup_{\substack{x \in [z_0,\ldots,z_{n-1}] \\ y \in [\kappa z_{n-1},\ldots,\kappa z_{0}]}}
\exp(f^n(x)-f^n(y)) \leq D_n.
\]

The following is the main result of Stadlbauer \cite{Stad},
restricted to the case where the base is a (finite state) subshift of finite
 type. We
 will use this in subsequent 
arguments.
(More generally, Stadlbauer considers skew product expansions of countable state
 Markov shifts.)
 
 \begin{proposition} [Stadlbauer \cite{Stad}, Theorem 4.1 and Theorem 5.6] \label{stadlbauer}
 Let $\widetilde \sigma : \Sigma_A^+ \times G \to \Sigma_A^+ \times G$ be a transitive
 symmetric skew-product extension of a mixing subshift of finite type 
 $\sigma : \Sigma_A^+ \to \Sigma_A^+$ 
 by a countable group $G$. Let $f : \Sigma_A^+ \to \mathbb R$
 be a weakly symmetric H\"older continuous function and define 
 $\widetilde f : \Sigma_A^+ \times G \to \mathbb R$
 by $\widetilde f(x,g) = f(x)$. Then $P_G(\widetilde \sigma,\widetilde f) = P(\sigma,f)$ 
 if and only if $G$ is amenable.
 \end{proposition}
 
 \begin{remark}
 In \cite{Stad}, Stadlbauer considers skew products with $\psi$ depending on only one co-ordinate. However, replacing $S$ by $\mathcal W^2$, one can easily recover the above formulation.
 \end{remark}

We end this section by discussing two-sided subshifts of finite type
and suspended flows over them. Given a finite alphabet $S$ and transition matrix
$A$, we define
\[
\Sigma_A = \left\{x = (x_n)_{n=0}^\infty \in S^{\mathbb Z} \hbox{ : }
A(x_n,x_{n+1})=1 \ \forall n \in \mathbb Z\right\}
\]
and the (two-sided) shift of finite type 
$\sigma : \Sigma_A \to \Sigma_A$ by
$(\sigma x)_n = x_{n+1}$. As before, we give $\Sigma_A$
with the product topology induced by the discrete topology on $S$ and this is compatible
with the metric $d(x,y)= 2^{-n(x,y)}$, where
\[
n(x,y) = \inf\{|n| \hbox{ : } x_n \neq y_n\},
\]
with $n(x,y)=\infty$ if $x=y$. 
Then $\Sigma_A$ is compact and $\sigma$ is a homeomorphism.
There is an obvious one-to-one correspondence between the periodic points of 
$\sigma : \Sigma_A \to \Sigma_A$ and $\sigma : \Sigma_A^+ \to \Sigma_A^+$.
Furthermore, we may pass from H\"older functions on $\Sigma_A$ to H\"older
functions on $\Sigma_A^+$ in such a way that sums around periodic orbits are preserved.
More precisely, we have the following lemma, due originally to Sinai.

\begin{lemma} [\cite{PP}] \label{sinai}
Let $f : \Sigma_A \to \mathbb R$ be H\"older continuous. Then
there is a H\"older continuous function $f' : \Sigma_A^+ \to \mathbb R$ 
(with a smaller H\"older exponent) such that $f^n(x) =  (f')^n(x)$, whenever 
$\sigma^n x=x$.
\end{lemma}

We may also define suspended flows over $\sigma : \Sigma_A \to \Sigma_A$.
Given a strictly positive continuous function $r : \Sigma_A \to \mathbb R^+$, we define the $r$-suspension space 
\[
\Sigma_A^r = \{(x,s) \hbox{ : } x \in \Sigma_A, \ 0 \leq s \leq r(x)\}/\sim,
\]
where $(x,r(x)) \sim (\sigma x,0)$. The suspended flow $\sigma^r_t : X_A^r \to X_A^r$ is defined by
$\sigma^r_t(x,s) = (x,s+t)$ modulo the identifications. Clearly, there is a natural one-to-one
correspondence between periodic orbits for $\sigma^r_t : \Sigma_A^r \to \Sigma_A^r$ and periodic orbits for
$\sigma :  \Sigma_A \to \Sigma_A$, and a $\sigma^r$-periodic orbit is 
prime if and only if the corresponding $\sigma$-periodic orbit is prime.
Furthermore, if $\gamma$ is a closed $\sigma^r$-orbit corresponding to the
closed $\sigma$-orbit $\{x,\sigma x, \ldots, \sigma^{n-1}x\}$ then
the period of $\gamma$ is equal to $r^n(x)$.

\section{Coverings and Geodesic Flows}

As in the introduction, we shall write $M = X/\Gamma$, $\widetilde M = X/\Gamma_0$ and 
$G = \Gamma/\Gamma_0$. 
There is a natural dynamical system related to the geometry of $M$, namely
the geodesic flow on the unit-tangent bundle
$SM = \{(x,v) \in TM \hbox{ : } \|v\|_x =1\}$,
where $\|\cdot\|_x$ is the norm induced by the Riemannian structure
on $T_xM$. 
For future reference, we write $p : SM \to M$ for the projection.
The geodesic flow $\phi_t : SM \to SM$ is defined as follows.
Given $(x,v) \in SM$, there is a unique unit-speed geodesic
$\gamma : \mathbb R \to M$ with $\gamma(0)=x$ and
$\dot \gamma (0) =v$. We then define
$\phi_t(x,v) = (\gamma(t),\dot \gamma(t))$.

The non-wandering set $\Omega(\phi) \subset SM$ is defined to be the set of points 
$x \in SM$ with the property that for every open neighbourhood $U$ of $x$, there exists
$t >0$ such that $\phi_t(U) \cap U \neq \varnothing$. It can be characterised as the set of vectors
tangent to $C(\Gamma)/\Gamma \subset M$.

The restriction of the geodesic flow to its non-wandering set,
$\phi_t : \Omega(\phi) \to \Omega(\phi)$, is an example of a {\it hyperbolic flow}. 
A $C^1$ flow $\phi_t : \Omega \to \Omega$ is hyperbolic if
\begin{enumerate}
\item
 there is a continuous $D\phi$-invariant
splitting of the tangent bundle 
\begin{align*}
T_{\Omega}(SM) = E^0 \oplus E^s \oplus E^u,
\end{align*}
where $E^0$ is the line bundle tangent to the flow and where
there exists constants $C, c>0$ such that
\begin{itemize}
\item[(i)]
$\|D\phi_t v\| \leq Ce^{-ct} \|v\|$, for all $v \in E^s$ and $t >0$;
\item[(ii)]
$\|D\phi_{-t} v\| \leq Ce^{-ct} \|v\|$, for all $v \in E^u$ and $t >0$,
\end{itemize}
\item
$\phi_t : \Omega \to \Omega$ is transitive (i.e. it has a dense orbit),
\item
the periodic $\phi$-orbits are dense in $\Omega$, and
\item
there is an open set $U \supset \Omega$ such that $\Omega = \bigcap_{t \in \mathbb R} 
\phi_t(U)$.
\end{enumerate}

The manifold $\widetilde M$ is a regular $G$-cover of $M$ and we 
let $\pi : \widetilde M \to M$ denote the covering map.
The geodesic flow $\widetilde \phi_t : S\widetilde M \to S\widetilde M$ is defined in a 
similar way to the geodesic flow on $SM$.
We write $\widetilde p : S\widetilde M \to M$ for the projection.
The action of $G$ extends to the unit-tangent bundle $S\widetilde M$ 
by the formula $g(x,v)=(gx,Dg_xv)$ and induces a regular covering
$\pi : S\widetilde M \to SM$ which commutes with the two flows.
(The use of $\pi$ to denote both coverings should not cause any confusion.)

There is a natural one-to-one correspondence between (prime) periodic orbits for $\phi_t : \Omega(\phi)
\to \Omega(\phi)$
and (prime) closed geodesics on $M$, with the least period being
equal to the length of the closed geodesic. We will typically write $\gamma$ for either a closed geodesic or a periodic orbit and allow the context to distinguish them. We will write $l(\gamma)$
for the length (period) of $\gamma$. The number $h =h(M)$ defined in the introduction as the
exponential growth rate of the number of $\gamma$ with $l(\gamma) \leq T$ is also equal to the
topological entropy of $\phi$.

\section{Markov Sections and Symbolic Dynamics} 
A particularly useful aspect of hyperbolic flows is that they admit a description by finite state
symbolic dynamics. We shall outline this construction below.

Given $\epsilon>0$, we define the (strong) {\it local stable manifold} $W^{s}_\epsilon(x)$ and
(strong) {\it local unstable manifold} $W^{u}_\epsilon(x)$ for a point $x \in SM$ by
\[
W^{s}_\epsilon(x) = \left\{y \in SM \hbox{ : } d(\phi_t(x),\phi_t(y)) 
\leq \epsilon \ \forall t \geq 0 \hbox{ and } \lim_{t \to \infty} d(\phi_t(x),\phi_t(y)) =0\right\}
\]
and
\[
W^{u}_\epsilon(x) = \left\{y \in SM \hbox{ : } d(\phi_{-t}(x),\phi_{-t}(y)) \leq \epsilon \ \forall t \geq 0
\hbox{ and } \lim_{t \to \infty} d(\phi_{-t}(x),\phi_{-t}(y)) =0\right\}.
\]
Provided $\epsilon>0$ is sufficiently small, these sets are diffeomorphic to 
$(\dim M -1)$-dimensional
embedded disks. If $x$ and $y$ are sufficiently close then there is a unique 
$t \in [-\epsilon,\epsilon]$
such that $W^s_\epsilon(x) \cap W^u_\epsilon(\phi_t(y)) \neq \varnothing$ and, furthermore, this
intersection consists of a single point denoted $[x,y]$. This pairing $[\cdot,\cdot]$ is called the 
{\it local product structure}.

Let
$D_1,\ldots,D_k$ be a family of co-dimension one disks that form a local cross section to 
the flow and
let $\mathcal P$ denote the Poincar\'e map between them.
For each $i=1,\ldots,k$,
let $T_i \subset \mathrm{int}(D_i) \cap \Omega(\phi)$ be sets which are chosen to be 
{\it rectangles} in the sense that whenever $x,y \in T_i$ then $[x,y]\in T_i$ and {\it proper}
(i.e. $T_i = \overline{\mathrm{int}(T_i)}$ for each $i$). 
(Here and subsequently, the interiors are taken 
relative to $D_i$.)
We then say that $T_1,\ldots,T_k$
are {\it Markov sections} for the flow if
\begin{enumerate}
\item
for $x \in \mathrm{int}(T_i)$ with $\mathcal P \in \mathrm{int}(T_j)$ then
$\mathcal P(W^s(x,T_i)) \subset W^s(\mathcal P x,T_j)$, and
\item
for $x \in \mathrm{int}(T_i)$ with $\mathcal P^{-1}x \in \mathrm{int}(T_j)$ then 
$\mathcal P^{-1}(W^u(x,T_i)) \subset W^u(\mathcal P^{-1} x,T_j)$,
\end{enumerate}
where $W^s(x,T_i)$ and $W^u(x,T_i)$ denote the projections of $W^s_\epsilon(x)$ and 
$W^u_\epsilon(x)$ onto $T_i$, respectively.

The local product structure on $SM$ induces a local product structure, also
denoted $[\cdot,\cdot]$ on transverse sections by projecting along flow lines.
The rectangles $T_i$ may be chosen so that $T_i = [U_i,S_i]$, where 
$U_i$ and $S_i$ are closed subsets of 
local unstable and stable manifolds, respectively.
Associated to this, we have projection maps $\rho_i^u : T_i \to U_i$ and
$\rho_i^s : T_i \to S_i$.

\begin{proposition} [Bowen \cite{Bow}] \label{markov}
For all $\epsilon>0$, the flow has Markov sections $T_1,\ldots,T_k$ such that 
$\mathrm{diam}(T_i)<\epsilon$, for $i=1,\ldots,k$
and such that $\bigcup_{i=1}^k \phi_{[0,\epsilon]} T_i = \Omega(\phi)$.
\end{proposition} 

These sections may be chosen to reflect the time-reversal symmetry of the geodesic flow.

\begin{lemma}  [Adachi \cite{Ad}, Rees \cite{Rees}] \label{adachi}
The Markov sections $T_1,\ldots,T_k$ may be chosen so that there is a fixed point free 
involution $\kappa : \{1,\ldots,k\} \to \{1,\ldots,k\}$ such that
$A(\kappa j,\kappa i)=1$ if and only if $A(i,j)=1$. Furthermore, the involution is consistent with the time reversing involution:
$T_{\kappa i} = \iota(T_i)$, where $\iota : SM \to SM$ is the map $\iota(x,v)=(x,-v)$.
\end{lemma}

The Markov sections allow us to relate $\phi_t : \Omega(\phi) \to \Omega(\phi)$ to a suspended flow
over a mixing subshift of finite type, whose alphabet $\{1,\ldots,k\}$ corresponds to the Markov
sections $\{T_1,\ldots,T_k\}$.

\begin{proposition}[Bowen \cite{Bow}]\label{symdyn}
There exists a mixing subshift of finite type $\sigma : \Sigma_A \to \Sigma_A$, a strictly positive 
H\"older
continuous function $r : \Sigma_A \to \mathbb R^+$ and a map $\vartheta : \Sigma_A^r \to \Omega(\phi)$ such that
\begin{enumerate}
\item
$\vartheta$ is a semi-conjugacy (i.e. $\phi_t \circ \vartheta = \vartheta \circ \sigma_t^r$);
\item
$\vartheta$ is a surjection and is one-to-one on a residual set;
\item
$h(\phi)=h(\sigma^r) =h$.
\end{enumerate} 
\end{proposition}

Clearly, the fixed point free involution $\kappa : \{1,\ldots,k\} \to \{1,\ldots,k\}$ induces a
fixed point free involution 
$\overline \kappa : \Sigma_A \to \Sigma_A$, defined by 
$(\overline \kappa x)_n = \kappa x_{-n}$. Furthermore,
$\vartheta \circ \overline \kappa = \iota \circ \vartheta$.

The above coding does not give a one-to-one correspondence between periodic orbits for
$\phi$ and $\sigma^r$. This is overcome by the following result, which 
is Bowen's generalisation to flows of a result of Manning for diffeomorphisms
\cite{Mann}.
For a flow $\xi_t$, let $\nu(\xi,T)$ denote the number of prime periodic $\xi$-orbits of period $T$ and $N_\xi(T)$ the number of periodic $\xi$-orbits to period
at most $T$.

\begin{lemma}[Bowen \cite{Bow}] \label{bowenmanning}
There exist a finite number of subshifts of finite type $\sigma_j : \Sigma_j \to \Sigma_j$ and strictly positive 
H\"older
continuous functions $r_j : \Sigma_j \to \mathbb R^+$, $j = 1,\ldots,q$, such
that
\begin{enumerate}
\item
$h(\sigma^{r_j}) < h$, $j=1,\ldots,q$; 
\item 
\[
\nu(\phi,T) = \nu(\sigma^r,T) + \sum_{j=1}^q (-1)^{\eta_j} \nu(\sigma_i^{r_j},T),
\]
where $\eta_j \in \{0,1\}$, $j=1,\ldots,q$.
\end{enumerate}
\end{lemma}

\begin{corollary}\label{cortobowenmanning}
$N_\phi(T) = N_{\sigma^r}(T) +O(e^{h'T})$,
where
$h' := \max_{1 \leq j \leq q} h(\sigma^{r_j}) <h$.
\end{corollary}

 We end the section by noting the following result.
 
 \begin{lemma} [\cite{PP}]\label{pressureforr}
The entropy $h$ is the unique real number for which $P(\sigma,-hr)=0$.
\end{lemma}

\section{The Skew Product Extension}

In this section we will describe a skew product extension of (the one-sided version of)
the shift of finite
type introduced above,
which will serve to encode information about how orbits on $SM$ lift to $S\widetilde M$, and relate 
this construction to the result of Stadlbauer, Proposition \ref{stadlbauer}, stated above.

Choose $\epsilon_0>0$ sufficiently small that every open ball in $SM$ with diameter less than 
$\epsilon_0$
is simply connected.
Let $U \subset SM$ be such an open ball. 
Then
$\pi^{-1}(U) = \bigcup_{g \in G} g \cdot \widetilde U$, 
where $\widetilde U$  is a connected component of $\pi^{-1}(U)$.
Since we can choose the Markov sections $T_i$ to have arbitrarily small diameters, 
for each $i = 1,\ldots,k$, we can choose $U_i$ be an open ball of diameter less than $\epsilon_0$ containing 
$T_i$. As above, we may write
$\pi^{-1}(U_i) = \bigcup_{g \in G} g \cdot \widetilde U_i$ and
$\pi^{-1}(T_i) = \bigcup_{g \in G} g \cdot \widetilde T_i$, where 
$\widetilde T_i = \pi^{-1}(T_i) \cap \widetilde U_i$, and we may assume that this decomposition is chosen with $\widetilde \iota (\widetilde T_i) = \widetilde T_{\kappa i}$, where $\iota$ is the time-reversing involution $\widetilde \iota : S\widetilde M \to S \widetilde M$ given by $\widetilde \iota (x,v) = (x,-v)$.

We will use the notation
\[
\widetilde{\mathcal T} = \pi^{-1}(\mathcal T)
= \bigcup_{i=1}^k \bigcup_{g \in G} g \cdot \widetilde T_i.
\]
Notice that each lifted section $g \cdot \widetilde T_i$ is transverse to the flow
$\widetilde \phi_t : S\widetilde M \to S\widetilde M$. 
We write write $\widetilde{\mathcal P} : \widetilde{\mathcal T} \to \widetilde{\mathcal T}$ for
the Poincar\'e map.

\begin{lemma} \label{technicalclaim}
Suppose that $A(i,j)=1$. There is a unique $g = g(i,j) \in G$ such that for any
$x \in T_i \cap \mathcal P^{-1}(T_j)$,
any $\widetilde x \in \pi^{-1}(x)$, and any $h\in G$, 
if $\widetilde x \in   h\cdot \widetilde T_i$ then $\widetilde{\mathcal P}(\widetilde x) \in
hg \cdot \widetilde T_j$.
In addition, $g(\kappa j, \kappa i)=g(i,j)^{-1}$. 
\end{lemma}

\begin{proof}
We will begin be proving the existence and uniqueness of $g$.
Let $x_1,x_2 \in T_i \cap \mathcal P^{-1}(T_j)$ and
let $c_1$ and $c_2$ be the $\phi$-orbit segments from $c_1(0)=x_1$ and $c_2(0)=x_2$ to
$c_1(1)=\mathcal P(x_1)$ and $c_2(1)=\mathcal P(x_2)$. We will show that there is a unique 
$g \in G$ such that the unique lifts of $c_1$ and $c_2$ that begin in $\widetilde T_i$ both 
have terminal points in $g \cdot \widetilde T_j$. The statement for all $h\in G$ will follow by 
translating by the isometry $h\in G$.

Let $\widetilde c_1$ and $\widetilde c_2$ be lifts of $c_1$ and $c_2$ with 
$\widetilde c_1(0), \widetilde c_2(0)\in \widetilde T_i$. Having chosen the Markov partition to 
have sufficiently small diameters
and the flow times between rectangles to be sufficiently small, there is an open ball $U\subset SM$ of 
diameter less than $\epsilon_0$ containing $T_i, T_j, c_1$ and $c_2$. Let $\widetilde U$ be the 
connected component of $\pi^{-1}(U)$ containing $\widetilde c_1$. Then $\widetilde U \cap \widetilde 
T_i \neq \varnothing$ and $\widetilde U \cap g\cdot \widetilde T_j \neq \varnothing$ and so $\widetilde 
T_i \cup g\cdot \widetilde T_j \subseteq \widetilde U$. It follows that $\widetilde c_2$ is entirely 
contained in $\widetilde U$ and so we must have $\widetilde c_2(1)\in g\cdot T_j$ as required.

For the final part, we note that $\iota (\widetilde c_1)$ is an orbit segment from $g\cdot \widetilde 
T_{\kappa j}$ to $\widetilde T_{\kappa i}$. It follows from the uniqueness in the previous that $g(\kappa 
j, \kappa i)=g(i,j)^{-1}$.
\end{proof}

We use the preceding lemma to define a skew product extension of the one-sided
shift of finite type $\sigma : \Sigma_A^+ \to \Sigma_A^+$.
Define $\psi : \Sigma_A^+ \to G$ (depending on two co-ordinates) by
$\psi(x)=\psi(x_0,x_1)=g(x_0,x_1)$, where $g=g(x_0,x_1)$ is the unique element of $G$ 
given by Lemma \ref{technicalclaim}. Then the skew product
$\widetilde \sigma: \Sigma_A^+ \times G \to \Sigma_A^+ \times G$
is defined by
\[
\widetilde \sigma(x,g) = (\sigma x, g\psi(x)).
\]
Furthermore, part (2) of Lemma \ref{technicalclaim} shows that the skew product extension is {\it symmetric} (with respect to the involution $\kappa$), i.e. that
$\psi(\kappa j,\kappa i)
= \psi(i,j)^{-1}$.

We note the relationship between periodic orbits for the lifted geodesic flow and for the skew product.

\begin{lemma}
A  periodic $\phi$-orbit $\gamma$ in $SM$, corresponding to a periodic $\sigma$-orbit
$\tau = \{x,\sigma x,\ldots,\sigma^{n-1}x\}$, lifts to a periodic orbit on $S\widetilde M$ if
and only if $\psi_n(x) =1_G$.
\end{lemma}

\begin{proof}
We will treat $\vartheta(x)$ as the initial point on $\gamma$. Let $\widetilde \gamma$ be the 
lift of $\gamma$ which
starts in $\widetilde T_{x_0}$. By Lemma \ref{technicalclaim}, $\widetilde \gamma$ ends in 
$\psi_n(x) \cdot \widetilde T_{x_0}$ and is thus periodic if and only if $\psi_n(x) =1_G$.
\end{proof}

We shall apply Proposition \ref{stadlbauer} to the skew product extension
$\widetilde \sigma: \Sigma_A^+ \times G \to \Sigma_A^+ \times G$.
To do this, we need to establish that two further conditions are satisfied:
that $\widetilde \sigma$ is transitive and that $r$ is weakly symmetric. We start with transitivity.

\begin{lemma} 
If $G$ is not equal to $\pi_1(M)$ then
the map $\widetilde \sigma: \Sigma_A^+ \times G \to \Sigma_A^+ \times G$ is transitive.
\end{lemma}

\begin{proof}
If $G$ is not equal to $\pi_1(M)$ then the geodesic flow $\widetilde \phi_t : S
\widetilde M \to
S\widetilde M$ is transitive. A proof is given in \cite{Dalbo} (page 94) for the case where
$X = \mathbb H^2$ but the argument clearly generalizes. (See also \cite{Eb} for the case of
variable curvature when
 $\Gamma$ is co-compact.)
Let $x \in S\widetilde M$ be a point with dense $\widetilde \phi$-orbit. Without loss of generality
$x \in \widetilde{\mathcal T}$ and then
$\{\widetilde{\mathcal P}^nx\}_{n=-\infty}^\infty$ is dense in $\widetilde{\mathcal T}$.
Suppose that 
$\widetilde A((i_j,g_j),(i_{j+1},g_{j+1}))=1$, where $\widetilde A$ is the transition
matrix for $\Sigma_A^+ \times G$, for
$j=0,\ldots,n$. Then
\[
U  = \bigcap_{j=0}^n \widetilde{\mathcal P}^{-j}(\mathrm{int}(g_j \cdot \widetilde T_{i_j}))
\]
is non-empty and open in $\widetilde{\mathcal T}$. 
(Here $\mathrm{int}(g_j \cdot \widetilde T_{i_j})$ is taken with respect to the 
co-dimension one disk containing $g_j \cdot \widetilde T_{i_j}$.)
Since $x$ has dense $\widetilde{\mathcal P}$-orbit, $\widetilde{\mathcal P}^m x \in U$ for some
$m \in \mathbb Z$. 
Then $\widetilde{\mathcal P}^{m+j}(x) \in \mathrm{int}(g_j \cdot \widetilde T_{i_j})$
for $j=0,\ldots,n$. By definition, this implies that the $\widetilde \sigma$-orbit of 
$(\vartheta(\pi(x)),g_0) \in \Sigma_A^+ \times G$ 
(where $\vartheta(\pi(x))$ is identified with a point in the one-sided shift)
passes through the (arbitrary) cylinder  
$[(i_0,g_0),\ldots,(i_n,g_n)]$ and is thus dense in $\Sigma_A^+ \times G$. Therefore,
$\widetilde \sigma : \Sigma_A^+ \times G \to \Sigma_A^+ \times G$ is transitive. 
\end{proof}

Let $r : \Sigma_A \to \mathbb R$ be the H\"older continuous function defined by
Proposition \ref{symdyn}. By Lemma \ref{sinai}, there is a H\"older continuous function on
$\Sigma_A^+$, which we will abuse notation by continuing to call $r$, with the same sums 
around periodic orbits.

\begin{lemma} \label{risweaksymm}
For any $\xi \in \mathbb R$, the function $-\xi r : \Sigma_A^+ \to \mathbb R$ is weakly symmetric.
\end{lemma}

\begin{proof}
It suffices to show that $r$ is weakly symmetric.
Since $\sigma : \Sigma_A^+ \to \Sigma_A^+$ is mixing, there exists $N \geq 0$ such that,
for each length $n$ cylinder 
$[\underline z] = [z_0,\ldots,z_{n-1}]$, we may find a periodic point $x \in [\underline z]$
of period $n+N$. Writing $x = (x_0,\ldots,x_{n+N-1},x_0,\ldots)$, we set
$\kappa x = (\kappa x_{n+N-1},\ldots,\kappa x_0,\kappa x_{n+N-1},\ldots)$. Clearly,
$\sigma^N(\kappa x)$ is a periodic point of period $n+N$ and 
$\sigma^N(\kappa x) \in [\kappa \underline z]$. Furthermore,
$r^{n+N}(x) = l(\gamma)$, for some $\phi$-periodic orbit $\gamma$
and
\[
r^{n+N}(\sigma^N(\kappa x)) = r^{n+N}(\kappa x) = l(\iota \gamma) = l(\gamma),
\]
where $\iota \gamma$ is the time-reversed periodic orbit corresponding to $\gamma$.
We therefore have
\begin{align*}
\exp(r^n(x)-r^n(\sigma^N(\kappa x)))
&= \exp((l(\gamma)-r^N(\sigma^n x)) - (l(\gamma)-r^N(\sigma^{n+N}(\kappa x)))) \\
&= \exp (r^N(\sigma^{n+N}(\kappa x)) - r^N(\sigma^n x)) \leq \exp(2N \|r\|_\infty),
\end{align*}
for some constant $C>0$.

Now let $x' \in [\underline z]$ and $y' \in [\kappa \underline z]$ be arbitrary.
We have
\begin{align*}
\exp(r^n(x')-r^n(y')) 
&= \exp(r^n(x)-r^n(\sigma^N(\kappa x)))
\frac{\exp(r^n(x')-r^n(x))}{\exp(r^n(y')-r^n(\sigma^N(\kappa x)))} \\
&\leq \exp(NC) \exp(2c/(1-2^{-\alpha})),
\end{align*}
where $r$ satisfies the H\"older condition $|r(x)-r(y)|\leq cd(x,y)^\alpha$. This completes the proof. 
\end{proof}

\section{Zeta functions}

In this section we shall prove that the equality of $h=h(M)$ and $h(\widetilde M)$ is equivalent 
to amenability of $G$.
To do this, we 
need to relate the growth of closed geodesics in $\mathcal C(\widetilde M,W)$
or, equivalently, of periodic $\widetilde \phi$-orbits which intersect $p^{-1}(W)$,
 to the Gurevi\v{c}
pressure. 
To do this, we make a particular choice of $W$, setting
$
W = \bigcup_{i=1}^k \mathrm{int}(\widetilde R_i),
$
where $\widetilde R_i$ is the thickened Markov section 
\[
\widetilde R_i =\{\widetilde \phi(x) \hbox{ : } x \in \widetilde T_i, \ 0 \leq t \leq \epsilon\}
\]
and $0 < \epsilon \leq \inf r$.
We now define a zeta function, analogous to the usual zeta function
for a flow but associated to
$\mathcal C(\widetilde M,W)$, by
\[
\zeta(s) = \prod_{\gamma \in \mathcal{C}'(\widetilde M,W)}
(1-e^{-sl(\gamma)})^{-1}
= \exp \sum_{m=1}^\infty \sum_{\gamma \in \mathcal C'(\widetilde M,W)} 
\frac{e^{-sml(\gamma)}}{m},
\]
where $\mathcal C'(\widetilde M,W)$ denotes the prime closed geodesics in 
$\mathcal C(\widetilde M,W)$.
This has abscissa of convergence $\widetilde h := h(\widetilde M)$.
A similar function may be defined using the set $\mathcal P'$ 
of prime periodic $\widetilde \sigma$-orbits  which intersect
$\Sigma_A^+ \times \{1_G\}$:
\[
Z(s) = \prod_{\tau \in \mathcal P'}
(1-e^{-s\lambda(\tau)})^{-1},
\]
where, for $\tau = \{(x,g),\widetilde \sigma(x,g),\ldots,\widetilde \sigma^{n-1}(x,g)\}$,
$\lambda(\tau) = r^n(x)$.
(One can, of course, describe this in terms of a suspended semi-flow over 
$\Sigma_A^+ \times G$ but this would make the notation more cumbersome.)
A standard calculation gives
\[
Z(s) = \exp \sum_{n=1}^\infty \frac{1}{n}
\sum_{(x,g) \in \mathcal P_n}
e^{-sr^n(x)},
\]
where
\[
\mathcal P_n =\{(x,g)  
\hbox{ : } \widetilde \sigma^n(x,g) =(x,g) \hbox{ and }
\widetilde \sigma^m(x,g) \in \Sigma_A^+ \times \{1_G\} \hbox{ for some } 0 \leq m <n\}.
\]
It is this last function that will be related to Gurevi\v{c} pressure.

The next lemma follows immediately from Lemma \ref{bowenmanning}.
In particular, the discrepancy between the number of
periodic $\widetilde \phi$-orbits with $l(\gamma) \leq T$
and the number of
periodic $\widetilde \sigma$-orbits with  $\lambda(\tau) = r^n(x) \leq T$ is at worst 
$O(e^{h'T})$.

\begin{lemma}
$\zeta(s)/Z(s)$ is analytic and non-zero for $\mathrm{Re}(s)> h'$.
\end{lemma}

\begin{corollary} \label{comparezetaZ}
$\zeta(s)$ has abscissa of convergence $h$ if and only if $Z(s)$ has abscissa 
of convergence $h$.
\end{corollary}

Let
\[
\mathcal Q_n
=\{x \hbox{ : } \sigma^nx=x \hbox{ and } \psi_n(x)=1_G\}.
\]

\begin{lemma} \label{comparison}
For all $n \geq 1$, 
$
\#\mathcal Q_n \leq \#\mathcal P_n \leq n \#\mathcal Q_n.
$
Hence, for $s \in \mathbb R$,
\[
\sum_{\substack{\sigma^n x=x \\ \psi_n(x)=1_G}}
e^{-sr^n(x)} 
\leq
\sum_{\substack{\widetilde \sigma^n(x,g) = (x,g) \\ \exists
0 \leq m < n \hbox{ : } \widetilde \sigma^m(x,g) \in \Sigma_A^+ \times \{1_G\}}}
e^{-sr^n(x)}
\leq
n
\sum_{\substack{\sigma^n x=x \\ \psi_n(x)=1_G}}
e^{-sr^n(x)} .
\]
\end{lemma}

\begin{proof}
Since $\widetilde \sigma^n(x,g) = (x,g)$ if and only if 
$\sigma^nx=x$ and $\psi_n(x)=1_G$, the first inequality follows by considering the the injection
$\mathcal Q_n \to \mathcal P_n : x \mapsto (x,1_G)$.
On the other hand, when
$\widetilde \sigma^n(x,g) =(x,g)$, the condition 
$\widetilde \sigma^m(x,g) \in \Sigma_A^+ \times \{1_G\}$
is equivalent to $\psi_m(x) g = 1_G$, i.e. $g = \psi_m(x)^{-1}$.
So, for each $\sigma^nx=x$ with $\psi_n(x)=1$, $\{g \in G \hbox{ : } (x,g) \in \mathcal P_n\} \subset
\{1,\psi(x)^{-1},\ldots,\psi_{n-1}(x)^{-1}\}$ and hence has cardinality at most $n$.
This proves the second inequality.
\end{proof}

Now, as promised, we relate the abscissa of convergence of $Z(s)$ to
Gurevi\v{c} pressure.

\begin{lemma} \label{aocforZ}
The abscissa of convergence of $Z(s)$ is the unique real number $\xi$ for which
$P_G(\widetilde \sigma,-\xi r)=0$.
\end{lemma}

\begin{proof}
It follows from Lemma \ref{comparison} that
\[
\sum_{x \in \mathcal Q_n}
e^{-sr^n(x)} 
\leq
\sum_{(x,g) \in \mathcal P_n}
e^{-sr^n(x)}
\leq
n
\sum_{x \in \mathcal Q_n}
e^{-sr^n(x)} 
\]
and hence that
\[
\limsup_{n \to \infty} \frac{1}{n} \log 
\sum_{(x,g) \in \mathcal P_n}
e^{-sr^n(x)}
=
\limsup_{n \to \infty} \frac{1}{n} \log \sum_{x \in \mathcal Q_n}
e^{-sr^n(x)}.
\]
Since
\[
P_G(\widetilde \sigma,-\xi r) = \limsup_{n \to \infty} \frac{1}{n} \log \sum_{x \in \mathcal Q_n}
e^{-\xi r^n(x)},
\]
we have that $Z(\xi)$ converges if $P_G(\widetilde \sigma,-\xi r)<0$ and diverges if 
$P_G(\widetilde \sigma,-\xi r)>0$.

The Gurevi\v{c} pressure is convex and hence continuous (\cite{sarignotes}, Proposition 4.4).
Choose $\xi,\xi' \in \mathbb R$ with $\xi< \xi'$ Write $r_0 = \inf \{r^n(x)/n \hbox{ : } \sigma^n x =x, \
n \geq 1\}>0$.
Then, for $\sigma^n x=x$,
$
e^{-\xi'r^n(x)} \leq e^{-\xi r^n(x)} e^{-n(\xi'-\xi) r_0}$.
Thus
\begin{align*}
P_G(\widetilde \sigma,-\xi' r) 
&= \limsup_{n \to \infty} \frac{1}{n} \log
\sum_{x \in \mathcal Q_n} e^{-\xi' r^n(x)}
\\
&\leq \limsup_{n \to \infty} \frac{1}{n} \log
\left( e^{-n(\xi' -\xi) r_0} \sum_{x \in \mathcal Q_n} e^{-\xi r^n(x)}\right)
\\
&= -(\xi' -\xi) r_0 + P_G(\widetilde \sigma,-\xi r) < P_G(\widetilde \sigma, -\xi r),
\end{align*}
so that $P_G(\widetilde \sigma,-\xi r)$ is a strictly decreasing function. 
Furthermore, the transitivity of 
$\widetilde \sigma : \widetilde \Sigma^+ \to \widetilde \Sigma^+$ ensures that 
$P_G(\widetilde \sigma,-\xi r)$ is not everywhere $-\infty$. Hence there is a unique
$\xi \in \mathbb R$ such that $P_G(\widetilde \sigma,-\xi r)=0$. 
By the above characterisation, this is the abscissa 
of convergence of $Z(s)$.
\end{proof}

We may now prove our main result, formulated for closed geodesics.

\begin{theorem}\label{mainresultforclosedgeodesics} 
Let $\Gamma$ be a convex co-compact group of isometries of a pinched Hadamard manifold
$X$ and let $\Gamma_0$ be a normal subgroup of $\Gamma$. Then $h(X/\Gamma_0) = h(X/\Gamma)$ if and only if
$G = \Gamma/\Gamma_0$ is amenable.
\end{theorem}

\begin{proof}
By Lemma \ref{pressureforr}, we have $P(\sigma, -hr)=0$ and, by Proposition \ref{stadlbauer}, 
$P_G(\widetilde \sigma,-hr) < P(\sigma,-hr)$ unless $G$ is amenable, in which case equality holds.
Hence, if $G$ is amenable then $P(\widetilde \sigma,-hr) =0$ and so $\widetilde h = h$.
On the other hand, if $G$ is not amenable then 
$P_G(\widetilde \sigma,-\xi r)=0$ for some $\xi<h$ and so, by Corollary \ref{comparezetaZ}
and Lemma \ref{aocforZ}, $\widetilde h <h$.
\end{proof}

\begin{remark}
We could also have proved that equality of critical exponents implies amenability directly by replacing Stadlbauer's result with a recent result  of Jaerisch
\cite{Jaerisch},
in which the Gurevi\v{c} pressure is replaced 
 by the logarithm of the spectral radius of a transfer operator associated to 
 $\widetilde \sigma$ acting on a suitably chosen Banach space,
together with some approximation arguments along the lines of those
used in \cite{Poll, PS94}.
\end{remark}

\end{document}